\documentclass[11pt]{article}

\usepackage{amsmath,amsthm}
\usepackage{amssymb}

\usepackage{enumitem}

\usepackage[latin1]{inputenc}

\newtheorem{theorem}{Theorem}[section]
\newtheorem{corollary}[theorem]{Corollary}

\newtheorem{proposition}[theorem]{Proposition}

\theoremstyle{definition}
\newtheorem{definition}[theorem]{Definition}
\newtheorem{remark}[theorem]{Remark}
\newtheorem{example}[theorem]{Example}

\numberwithin{equation}{section}

\newcommand{\Naturali}{{\mathbb{N}}}
\newcommand{\Complessi}{{\mathbb{C}}}
\newcommand{\Toro}{{\mathbb{T}}}
\newcommand{\I}{{\mathcal I}}
\newcommand{\Lc}{{\mathcal L}}
\newcommand{\U}{{\mathcal U}}

\begin{document}

\title{The Fourier-Stieltjes algebra of a $C^*$-dynamical system II}

\author{Erik B\'edos,
Roberto Conti\\}

\date{\today}

\maketitle
\markboth{Erik B\'edos, Roberto Conti}{
}

\renewcommand{\sectionmark}[1]{}
\begin{abstract}
 We continue our study of the Fourier-Stieltjes algebra associated to a twisted (unital, discrete) C*-dynamical system and discuss how the various notions of equivalence of such systems are reflected at the algebra-level. 
As an application, we show that the amenability of a system, as defined in our previous work, is preserved under Morita equivalence.


\end{abstract}

\section{Introduction} \label{Intro}

The classical notion of Fourier-Stieltjes algebra of a locally compact group $G$ \cite{Eym} was extended
in \cite{BeCo6} to a (unital, discrete) twisted C*-dynamical system $\Sigma = (A, G, \alpha, \sigma)$. In short, the outcome is a Banach algebra $B(\Sigma)$ attached to $\Sigma$ 
with a rich analytical structure
that can be better described in terms of coefficients of the so-called equivariant representations of $\Sigma$. In the case where $A$ is trivial, any such a representation is nothing but a unitary representation of $G$ on a Hilbert space, and one therefore recovers the  Fourier-Stieltjes algebra $B(G)$. Some aspects of the classical theory survive to the new setting, notably
the inclusion of $B(\Sigma)$ in the completely bounded full/reduced multipliers of $\Sigma$, as well as the fact that $B(\Sigma)$ is spanned by the $\Sigma$-positive definite functions, which themselves give rise to completely positive maps of the full and reduced twisted crossed product $C^*$-algebras associated to $\Sigma$. We note here that in the case of an untwisted system our concept of  $\Sigma$-positive definiteness can be reformulated using the notion of  completely positive Herz-Schur $\Sigma$-multiplier (cf.~\cite{MSTT}). 
We also recall that  $B(G)$ continuously embeds into $B(\Sigma)$, although these two algebras differ significantly  from each other
for it can be shown that, under mild assumptions, $B(\Sigma)$ is always noncommutative (actually, $B(G)$ is contained in the center of $B(\Sigma)$).
Finally, we mention that one can use the aforementioned coefficients of equivariant representations of $\Sigma$ to introduce suitable approximation properties for $\Sigma$, such as amenability (cf.~\cite{BeCo6}) and the Haagerup property (cf.~\cite{MSTT}),
that parallel the analogous notions for $G$ and provide intrinsic features of the dynamical system $\Sigma$.

The main motivation for this paper was to explore to which extent the Fourier-Stieltjes algebra $B(\Sigma)$ depends on $\Sigma$. We recall that if $G_1$ and $G_2$ are locally compact groups, then Walter showed in \cite{Wal} that $B(G_1)$ and $B(G_2)$ are isometrically isomorphic as Banach algebras if and only if $G_1$ and $G_2$ are topologically isomorphic. Hence one may hope that $B(\Sigma)$ is better suited to characterize $\Sigma$ than other algebras associated  to it. 
Now there are several natural notions of equivalence between two dynamical systems  $\Sigma = (A, G, \alpha, \sigma)$ 
and  $\Theta = (B, H, \beta, \theta )$, most notably exterior equivalence, conjugacy, and cocycle conjugacy, but also Morita equivalence (in the case where $G=H$). It is immediate that the first two notions are stronger than the third one,
which is itself stronger than the last one.
We show in Theorem \ref{coc-conj-isom} that $B(\Sigma)$ and $B(\Theta)$ are isometrically isomorphic whenever $\Sigma$ and $\Theta$ are cocycle conjugate (up to a group isomorphism), 
in a way that preserves 
the classical Fourier-Stieltjes algebras of the corresponding groups, and also the canonical copies of the corresponding algebras.        
In connection with this result, we also note that the Fourier-Stieltjes algebra of a system does not detect a perturbation of the system by a  $\Toro$-valued group $2$-cocycle, cf.~Remark \ref{coc-pert2}.
In the case of Morita equivalent systems, the connection between the Fourier-Stieltjes algebras remains somewhat more elusive, but we are at least able to show that these algebras can be determined from each other, see Corollary \ref{M-eq-FS}.
However, as a byproduct of this study, we obtain an interesting consequence for Morita equivalent systems, namely we show in Theorem \ref{Amen} that the amenability of a system  (as defined in \cite{BeCo6}) is preserved
under such an equivalence.

 The paper is organized as follows. After some preliminaries in Section 2, we review in Section 3 some of the natural notions of equivalence for twisted $C^*$-dynamical systems (exterior equivalence, (group) conjugacy, and cocycle (group) conjugacy) and prove that the Fourier-Stieltjes algebra is invariant, up to isometric isomorphism, under cocycle group conjugacy (which is the most general among  these notions). In Section 4 we consider two Morita equivalent systems and point out that there is, up to isomorphism, a one-to-one correspondence between the equivariant representations of the respective systems. We use this to show that the corresponding Fourier-Stieltjes algebras can then be recovered from each other. Finally, in Section 5, we recall our definition of amenability for a system and show that this property is Morita invariant.      

\section{Preliminaries}\label{Preliminaries} 

We only consider \emph{unital}  $C^*$-algebras in this paper, and a homomorphism 
between two such algebras will always mean a unit preserving $*$-homomorphism.  Isomorphisms and automorphisms between $C^*$-algebras are therefore 
 also assumed to be $*$-preserving. The group of unitary elements in a $C^*$-algebra $A$ will be denoted by $\U(A)$, the center of $A$ by $Z(A)$, and the group of automorphisms of $A$  by ${\rm Aut}(A)$. The identity map on $A$ will be denoted by ${\rm id}$ (or ${\rm id}_A$). 
 If $B$ is another $C^*$-algebra, $A\otimes B$ will denote their minimal tensor product.

By a Hilbert $C^*$-module, we will always mean a  \emph{right} Hilbert $C^*$-module, unless otherwise specified, and follow the notation introduced in \cite{La1}. In particular, all inner products will be assumed to be linear in the second variable, $\Lc_B(X, Y)$  will denote the space of all adjointable operators between two Hilbert $C^*$-modules $X$ and $Y$ over a $C^*$-algebra $B$, and $\Lc_B(X) = \Lc_B(X,X)$. A representation of a $C^*$-algebra $A$ on a Hilbert $B$-module $Y$ is then a homomorphism from $A$ into the $C^*$-algebra $\Lc_B(Y)$. If $Z$ is another Hilbert $C^*$-module (over $C$), we will let $\pi \otimes \iota : A  \to \Lc_{B\otimes C}(Y \otimes Z)$ denote the amplified representation of $A$ on $Y \otimes Z$ given by $(\pi \otimes \iota)(a) = \pi(a) \otimes I_Z$, where the Hilbert $B\otimes C$-module $Y \otimes Z$ is the external tensor product of $Y$ and $Z$ (cf.~\cite{La1}), and $I_Z$ denotes the identity operator on $Z$.  If $Z$ is a Hilbert space,  then we  consider $Y\otimes Z$ as a Hilbert $B$-module.

The quadruple $\Sigma = (A, G, \alpha,\sigma)$ will always denote a \emph{twisted 
unital discrete
$C^*$-dynamical system}.  This means that
$A$ is a  $C^*$-algebra with unit $1_A$, 
$G$ is a discrete group with identity $e$
and $(\alpha,\sigma)$ is a \emph{twisted
action} of $G$ on $A$ (sometimes called a cocycle $G$-action on $A$), that is,
$\alpha$ is a map from $G$ into ${\rm Aut}(A)$ 
and  $\sigma: G \times G \to \U(A)$ is a normalized $2$-cocycle for $\alpha$, such that
\begin{align*}
\alpha_g  \alpha_h & = {\rm Ad}(\sigma(g, h))   \alpha_{gh}, \\
\sigma(g,h) \sigma(gh,k) & = \alpha_g(\sigma(h,k)) \sigma(g,hk), \\
\sigma(g,e) & = \sigma(e,g) = 1_A  
\end{align*}
for all $g,h,k \in G$. Of course,  ${\rm Ad}(u)$ denotes here the (inner) automorphism  of $A$ implemented by the unitary $u$ in $\U(A)$.
If $\sigma=1$ is the trivial $2$-cocycle, that is, $\sigma(g,h)=1_A$ for all $g,h \in G$, then $\alpha$ is a genuine action and $\Sigma$ is an ordinary $C^*$-dynamical system (see e.g.\ \cite{Wi, BrOz}), usually denoted by $\Sigma=(A, G, \alpha)$.  If $\sigma$ is \emph{central}, that is, it takes values in  $\U(Z(A))$, then $\alpha$ is also a genuine action of $G$ on $A$, and  this is the case studied in \cite{ZM}. In the sequel we will often just use the word system to mean a discrete unital twisted $C^*$-dynamical system.

An \emph{equivariant representation} 
of $\Sigma$ on a Hilbert $A$-module $X$ (see e.g.~\cite{BeCo3, BeCo4}) 
is  a pair $(\rho, v)$ where 
 $\rho : A \to \Lc_A(X)$ is a representation of $A$ on $X$ and  $v$ is a map from $G$ into the group $\mathcal{I}(X)$  of all $\mathbb{C}$-linear, invertible, bounded maps from $X$ into itself, which  satisfy:
\begin{itemize}
\item[(i)]  \quad $\rho(\alpha_g(a))  = v(g)  \rho(a)  v(g)^{-1}, \quad  \quad g\in G,   a \in A,$
\item[(ii)]  \quad$v(g) v(h)  = {\rm ad}_\rho(\sigma(g,h))  v(gh), \quad   \quad g, h \in G,$
\item[(iii)] \quad $\alpha_g\big(\langle x, x' \rangle\big)   = \langle v(g) x, v(g) x' \rangle , \quad  \quad g\in G,  x,  x' \in X,$ 
\item[(iv)]  \quad$v(g)(x \cdot a)  = (v(g) x)\cdot \alpha_g(a), 
\quad  \quad  g \in G, x\in X, a \in A$. 
\end{itemize}
In (ii) above, $ {\rm ad}_\rho(\sigma(g,h)) \in \mathcal{I}(X) $ is defined by
\[{\rm ad}_\rho(\sigma(g,h)) x = \big(\rho(\sigma(g,h)) x \big)\cdot \sigma(g,h)^*, \quad g, h \in G, x \in X. \]
Note that the equivariant representations of $\Sigma$ may  instead be presented in terms of $(\Sigma$,$\Sigma$)-compatible actions, as in \cite{EKQR-0, EKQR}, cf.~Remark \ref{equirep-action}. 
Note also that condition (iii) implies that each $v(g)$ is isometric.
% w.r.t.~the 

For completeness, we mention some examples of equivariant representations. 
First, the  \emph{trivial equivariant representation} of $\Sigma$, which is the pair $(\ell, \alpha)$ acting on $A$, considered as a right $A$-module over itself in the canonical way, where $\ell:A\to \Lc_A(A)$ is given by left-multiplication. 
Next, let $A^G := \ell^2(G,A)$ denote the right $A$-module given by 
\[A^G = \Big\{\xi:G\to A \mid  \sum_{g\in G}   \xi(g)^{*} \xi(g) \text{ is norm-convergent in $A$}\Big\},\] with the obvious right $A$-module structure, and inner product given by 
\[\langle \xi, \eta\rangle = \sum_{g\in G}  \xi(g)^* \eta(g). \]  Then the \emph{regular equivariant representation} of $\Sigma$ on $A^G$ is the pair 
$(\check{\ell}, \check{\alpha})$  acting on $A^G$ defined by
 \[(\check{\ell}(a) \xi)(h) = a \xi(h), \quad 
(\check{\alpha}(g)\xi)(h) =  \alpha_g(\xi(g^{-1}h)) \]
for $a \in A,  \xi \in A^G$ and $ g, h \in G$.

More generally, if $(\rho, v)$ is an equivariant  representation of $A$ on a right Hilbert $A$-module $X$ and $w$ is a unitary representation of $G$ on some         
Hilbert space $\mathcal H$, then $(\rho \otimes \iota, v \otimes w)$ is an
equivariant representation of $\Sigma$ on $X \otimes {\mathcal H}$.

 One can also form the tensor product of equivariant representations. 
Assume that $(\rho_1,v_1)$ and $(\rho_2,v_2)$ are equivariant representations 
 of $\Sigma$ on some Hilbert $A$-modules 
$X_1$ and $X_2$, respectively.
We can then form the internal tensor product  $X_1\otimes_{\rho_2} Y$, which is a right Hilbert $A$-module (cf.~\cite{La1}); we will suppress $\rho_2$ in our notation and denote $X_1\otimes_{\rho_2} X_2$ by $X_1\otimes_A X_2$, as it is quite common in the literature. 
Then the tensor product
$(\rho_1,v_1) \otimes (\rho_2,v_2)$
acts on 
$X_1 \otimes_{A} X_2$ 
as follows.
For  $a \in A$, let $(\rho_1 \otimes \rho_2)(a)\in \Lc_A(X_1 \otimes_{A} X_2) $ be the map determined on simple tensors by
\[(\rho_1 \otimes \rho_2)(a) (x_1 \dot\otimes x_2) = \rho_1(a)x_1 \dot\otimes  x_2 \quad \text{for } x_1 \in X_1  \text{and } x_2 \in X_2. \]
Moreover, for every $g \in G$, let $(v_1 \otimes v_2)(g)$ in $\I(X_1 \otimes_{A} X_2)$ be the map determined on simple tensors by
\[(v_1 \otimes v_2)(g) (x_1 \dot\otimes x_2) = v_1(g)x_1 \dot\otimes v_2(g)x_2\quad  \text{for } x_1 \in X_1\text{  and } x_2 \in X_2 . \]
Then  $(\rho_1,v_1)\otimes(\rho_2,v_2):=(\rho_1 \otimes \rho_2, v_1 \otimes v_2)$ is an equivariant representation of $\Sigma$ on the right Hilbert $A$-module $X_1 \otimes_{A} X_2$ (cf.~\cite{EKQR, BeCo3}).

 Let  $(\rho,v)$ be an equivariant representation of $\Sigma$ on a Hilbert $A$-module $X$ and let $x, y \in X$.    
Then we define $T_{\rho,v,x,y}:G\times A \to A$ by
\[T_{\rho,v,x,y}(g,a) = \big\langle x,  \rho(a)  v(g)  y \big \rangle  \quad \text{for } a \in A,  g \in G,\] 
and think of $T_{\rho, v, x, y}$ as an  $A$-valued coefficient function associated with $(\rho, v)$. 

 The \emph{Fourier-Stieltjes algebra} $B(\Sigma)$
 is defined in \cite{BeCo6} as the collection of all 
the  maps from $G\times A$ into $A$
    of the form $T_{\rho, v, x, y}$ for some equivariant 
representation $(\rho,v)$ of $\Sigma$ on a Hilbert $A$-module $X$ and  $x, y \in X$. 
 Then $B(\Sigma)$  becomes a unital subalgebra of $L(\Sigma)$, where
 \[ L(\Sigma) =\{ T:G\times A \to A\mid T \text{ is linear in the second variable}\} \]
is equipped with its natural algebra structure:
for $T, T' \in L(\Sigma)$ and $\lambda \in \Complessi$,  we let 
$T+T',  \lambda T$, $ T\cdot T'$  and $I_\Sigma$ be the maps in $L(\Sigma)$  defined 
by 
\begin{align*}
(T+T')(g,a) & := T(g,a)+T'(g,a) \\
(\lambda T) (g,a) & := \lambda  T(g,a) \\
(T \cdot T')(g,a) & := T(g,T'(g,a)) \\
I_\Sigma(g,a) & := a
\end{align*}
for $g\in G$ and $a \in A$. Given $T\in L(\Sigma)$ and $g \in G$, we will sometimes write $T_g$ for the linear map from $A$ into itself given by $T_g(a)=T(g, a)$ for all $a\in A$.

If $T\in B(\Sigma)$, letting  $\|T\|$ denote the infimum of  the set of values $\|x\| \|y\|$ associated with the possible decompositions of $T$ of the form  $T = T_{\rho, v, x,y}$, one gets a norm on $B(\Sigma)$ such that   
 $B(\Sigma)$ is a unital Banach algebra w.r.t.~$\|\cdot\|$.
 
We also recall that there is a canonical way of embedding $B(G)$ into $B(\Sigma)$ (cf.~\cite[Proposition 3.2]{BeCo6}):
 For $f \in B(G)$, define $T^f \in L(\Sigma)$ by $T^f(g, a) = f(g)  a$ for $g\in G$ and $a \in A$. 
Then $T^f \in B(\Sigma)$, and the map $f \to T^f $ gives an injective, contractive, algebra-homomorphism of $B(G)$ into $B(\Sigma)$.

The Fourier-Stieltjes algebra $B(\Sigma)$ also contains a copy of  $A$. Indeed, for $b \in A$, let $T^b \in L(\Sigma)$ be given by $T^b(g,a) = ba$ for all $g\in G$ and $a\in A$. Then we have that $T^b = T_{\ell, \alpha, b^*\!, 1_A} \in B(\Sigma)$ and $\|T^b\| \leq \|b\|$. From this, one readily deduces that the map $b\to T^b$ gives an 
isometric algebra-homomorphism from $A$ into $B(\Sigma)$.

Finally, we recall that, as in the classical case, $B(\Sigma)$ is spanned by its positive definite elements (cf.~\cite[Corollary 4.5]{BeCo6}). For the ease of the reader, we review how positive definiteness is defined in our setting.  Let $T\in L(\Sigma)$. 
Then $T$ is {called \it positive definite} (w.r.t.~$\Sigma$), or \emph{$\Sigma$-positive definite}, when for any $n\in \Naturali$, $g_1, \ldots, g_n \in G$ and $a_1, \ldots, a_n \in A$, the matrix
  \[\Big[ \alpha_{g_i}\Big(T_{g_i^{-1}g_j}\big(\alpha_{g_i}^{-1}\big(a_i^*a_j\sigma(g_i, g_i^{-1}g_j)^*\big)\big)\Big) \sigma(g_i, g_i^{-1}g_j)\Big]\]
is positive in $M_n(A)$ (the $n\times n$ matrices over $A$). 
As shown in \cite[Corollary 4.4]{BeCo6}, which is an analogue of the Gelfand-Raikov theorem, this is  equivalent to requiring that
 $T$ may be written as $T = T_{\rho, v, x, x}$ for some equivariant representation $(\rho, v)$ of $\Sigma$ on some Hilbert $A$-module $X$ and some $x\in X$. It then follows that 
 \[\|T\|_\infty:=\sup \{\|T_g\|\mid g\in G\} = \|T_e(1_A) \|=  \|\langle x, x\rangle_A\|\] (cf. \cite[Corollary 4.3]{BeCo6}).
We set 
\[P(\Sigma) =\big\{T\in L(\Sigma)\mid T  \text{is positive definite $($w.r.t. }  \Sigma )\big\}.\]

\section{Cocycle group conjugate systems}

There are various notions of equivalence for $C^*$-dynamical systems in the literature.
In this section we will study how the notions of exterior equivalence, (group) conjugacy and cocycle (group) conjugacy  are reflected at the level of the Fourier-Stieltjes algebras. 

\begin{definition}
Consider a system $\Sigma = (A, G, \alpha, \sigma)$,  
and let $w:G\to \mathcal{U}(A)$ be a normalized map, that is, such that $w(e)=1_A$. Then it is well known (cf.~\cite[Section 3]{PaRa}) that we get another twisted action $(\alpha^w, \sigma^w)$ of $G$ on $A$ by setting
\[\alpha_g^w = {\rm Ad}(w(g)) \circ \alpha_g \quad \text{and} \quad \sigma^w(g,g') = w(g)\alpha_g(w(g'))\sigma(g,g') w(gg')^*\]
for all $g, g' \in G$.  We then set $\Sigma^w:= (A, G, \alpha^w, \sigma^w)$ and call $\Sigma^w$ a \emph{perturbation of $\Sigma$ by $w$}. 
\end{definition}

\begin{remark} \label{coc-pert} 
Another way to perturb a system $\Sigma = (A, G, \alpha, \sigma)$ 
is as follows. Let $\alpha'$ denote the restriction of $\alpha$ to a (genuine) action of $G$ on $Z(A)$, 
and let $\eta : G\times G \to \mathcal{U}(Z(A))$ be a normalized $2$-cocycle for $\alpha'$. 
(For example, 
 we can let $\eta : G\times G \to \Toro$ be any normalized $2$-cocycle for the group $G$ and consider $\eta$ as a $2$-cocycle for $\alpha'$.)
Then we get a twisted action $(\alpha , \sigma_\eta )$ of $G$ on $A$ by setting 
\[ \sigma_\eta (g, g') := \sigma(g,g')\eta (g, g')\]
for all $g, g' \in G$. The system $\Sigma(\eta) := (A, G, \alpha, \sigma_\eta )$ is called a \emph{perturbation of $\Sigma$ by $\eta $}. 
\end{remark} 

\begin{definition} 
Two systems $\Sigma = (A, G, \alpha, \sigma)$ and $\Theta =(A, G, \beta, \theta )$ 
 are called \emph{exterior equivalent}, and we write  $\Sigma \sim_e  \Theta$, when $\Theta  = \Sigma^w$ for some map $w:G\to \mathcal{U}(A)$ (which is then necessarily normalized). 
\end{definition}

\begin{example} \label{ext-eq} 
Let  $\alpha$ and $\beta$ be two genuine actions of $G$ on $A$ and set $\Sigma = (A, G, \alpha, 1)$ and $\Theta =(A, G, \beta, 1)$. We recall that a map $w: G\to \mathcal{U}(A)$ is called a \emph{$1$-cocycle for $\alpha$} when it satisfies that  $w(gg') = w(g)\alpha_g(w(g'))$ for all $g, g'\in G$. Then we have that  $ \Sigma  \sim_e \Theta$ if and only if there exists some $1$-cocycle $w: G\to \mathcal{U}(A)$ for $\alpha$ such that $\beta_g = {\rm Ad}(w(g)) \circ \alpha_g $ for all $g\in G$. One usually says that $\beta$ is a perturbation of $\alpha$ by $w$ in this case. 

Assume now that $\alpha$ and $\beta$
agree up to inner automorphisms, that is, they satisfy that $\beta_g = {\rm Ad}(u(g)) \circ \alpha_g$ for some map $u:G\to \mathcal{U}(A)$, which may be assumed to be normalized. Set \[\partial u(g, h):= u(g)\alpha_g(u(h))u(gh)^*\] for all $g, h \in G$. Then it can easily be checked that $\partial u$ is a $2$-cocycle for $\beta$ taking its values in $\mathcal{U}(Z(A))$. If $\partial u \neq 1$, i.e., $u$ is not a $1$-cocycle for $\alpha$, then we get that $(\beta, \partial u)$ 
is a twisted action of $G$ on $A$ satisfying  that 
$\Sigma =  (A, G, \alpha, 1) \sim_e (A, G, \beta, \partial u)$. Similarly, 
$\Theta \sim_e (A, G, \alpha, \partial u^*)$, where $u^*(g):=u(g)^*$ for all $g \in G$.

We note that if the map $u$ above takes its values in $ \mathcal{U}(Z(A))$ (so we have $\beta=\alpha$), and $\alpha'$ denotes the restriction of $\alpha$ to an action of $G$ on $Z(A)$, then $\partial u$ is a normalized $2$-cocycle for $\alpha'$ (called a \emph{coboundary} for $\alpha'$). A perturbation of $\Sigma$ by $u$ is then clearly the same as a perturbation of $\Sigma$ by $\partial u$ (in the sense of Remark \ref{coc-pert}), i.e., we have $ \Sigma^u = \Sigma(\partial u) $, and we get that $\Sigma \sim_e \Sigma(\partial u)$ in this case.
\end{example}
Next, consider $\Sigma = (A, G, \alpha, \sigma)$ and note that if $\phi:A\to B$ is an isomorphism of $C^*$-algebras and $\varphi:G\to H$ is an isomorphism of groups, then 
we get a new system $\Theta = (B, H, \beta, \theta )$ by setting
\[\beta_{h}= \phi \circ \alpha_{\varphi^{-1}(h)} \circ \phi^{-1} \quad \text{and} \quad
\theta (h, h') = \phi\big(\sigma(\varphi^{-1}(h), \varphi^{-1}(h'))\big)\]
for all $h, h' \in H$. This motivates the following notion.

\begin{definition} Two systems $\Sigma = (A, G, \alpha, \sigma)$ and $\Theta =(B, H, \beta, \theta )$ are said to be \emph{group conjugate} if there exist an 
isomorphism $\phi:A\to B$ and an  isomorphism $\varphi:G\to H$ such that
\begin{itemize}
\item[(i)] $\beta_{\varphi(g)}= \phi \circ \alpha_g \circ \phi^{-1}$,
\item[(ii)] $\theta \big(\varphi(g), \varphi(g')\big) = \phi\big(\sigma(g, g')\big)$
\end{itemize}
for all $g, g' \in G$, in which case we write $ \Sigma \sim_{gc}  \Theta$.
In the case where $H=G$, we will say that $\Sigma$ and $\Theta $ are \emph{conjugate},  and write $ \Sigma \sim_{c} \Theta$, if $\varphi$ can be chosen to be the identity map.
\end{definition}

\begin{definition} Two systems $\Sigma = (A, G, \alpha, \sigma)$ and $\Theta =(B, H, \beta, \theta )$ are said to be 
\emph{ cocycle group conjugate} if $\Sigma^w  \sim_{gc}   \Theta$ for some normalized $w:G\to \mathcal{U}(A)$, in which case we write $ \Sigma \sim_{cgc} \Theta $.
Equivalently, as one readily checks,  $  \Sigma \sim_{cgc} \Theta$ if and only if $\Theta $ is exterior equivalent to some 
group conjugate of $\Sigma$.
In the case where $H=G$, we will say that $\Sigma$ and $\Theta $ are \emph{cocycle conjugate}, and write $ \Sigma \sim_{cc} \Theta $, 
if  $\Sigma^w$ is conjugate to $\Theta $  for some normalized $w:G\to \mathcal{U}(A)$.
\end{definition}

Discarding  set-theoretical problems, one may show without much trouble that $\sim_{cgc}$ (resp.~$\sim_{cc}$) satisfies the properties of an equivalence relation. Moreover, it is evident from the definitions that (group) conjugacy
and exterior equivalence
are stronger notions than cocycle (group) conjugacy.

\begin{example}
\label{excocycony}
Assume again $\alpha$ and $\beta$ are genuine actions of $G$ on $A$. Then 
we have $(A, G, \alpha, 1) \sim_{cc}  (A, G, \beta, 1)$ if and only if $(A, G, \alpha^w, 1^w) \sim_{c} (A, G, \beta, 1)$ for some normalized $w:G\to\mathcal{U}(A)$, in which case we get 
$1 = 1^w(g,g')
= w(g) \alpha_g(w(g')) w(gg')^*$ for all $g\in G$, so that $w$ is a $1$-cocycle for $\alpha$. Hence $(A, G, \alpha, 1) \sim_{cc} (A, G, \beta, 1) $ if and only if there is a perturbation of $\alpha$ by a $1$-cocycle for $\alpha$ which is conjugate to $\beta$, i.e., $\alpha$  is cocycle conjugate to $\beta$  (as defined for example in \cite[II.10.3.18]{Bl}).
\end{example}

It is part of the folklore
that the $C^*$-crossed products associated to cocycle conjugate systems are isomorphic, both in the full and in the reduced case, 
via an isomorphism that preserves the ``diagonal'' algebra 
(for partial results in this direction, see e.g.~\cite[Lemma 3.2]{PaRa} and \cite[Lemma 2.68]{Wi}). 
In our setting, we have:

\begin{theorem} \label{coc-conj-isom}
Assume $\Sigma = (A, G, \alpha, \sigma)$ 
and  $\Theta =(B, H, \beta, \theta )$ are cocycle group conjugate. 
Then $B(\Sigma)$ and $B(\Theta )$ are isometrically isomorphic.

 More precisely, there exists an algebra-isomorphism $\Psi: B(\Theta) \to B(\Sigma)$ such that
\begin{itemize}
\item[1$)$] $\Psi$ is isometric$;$
\item[2$)$] $\Psi$ maps the copy of $B(H)$ inside $B(\Theta)$ isometrically onto the copy of $B(G)$ inside $B(\Sigma)$  $($w.r.t.~the norms of $B(G)$ and $B(H)$$)$$;$
\item[3$)$] $\Psi$ restricts to an isomorphism from the copy of $B$ inside $B(\Theta)$ onto the copy of $A$ inside $B(\Sigma)$, and the associated map from $B$ to $A$ is $*$-preserving $($hence isometric$)$.
\end{itemize}
\end{theorem}
\begin{proof}
It clearly suffices to prove the result in the two separate cases where $\Sigma$ and $\Theta$ are group conjugate or exterior equivalent.

 Assume first that $\Sigma \sim_{gc} \Theta $ via isomorphisms $\phi:A\to B$ and $\varphi:G\to H$. Then 
 the reader should have no trouble in verifying 
that the map
$\Psi: B(\Theta )\to B(\Sigma)$ given by
\[[\Psi(S)](g,a) = \phi^{-1}\big(S(\varphi(g), \phi(a))\big)\]
for $S \in B(\Theta )$, $g \in G$ and $ a\in A$, is a well-defined  algebra-isomorphism
satisfying 1), 2) and 3).

 Next, assume that $\Sigma$ and $\Theta$ are exterior equivalent, so we have $\Theta  = \Sigma^w$ for some normalized map $w:G\to \mathcal{U}(A)$, where  $\Sigma^w= (A, G, \alpha^w, \sigma^w)$. Noting that $L(\Sigma^w)=L(\Sigma)$, it is straightforward to check that the map
$\Pi: L(\Sigma)\to L(\Sigma^w)$ given by
\[[\Pi(T)](g,a) = T\big(g, aw(g)\big)w(g)^*\]
for $T \in L(\Sigma)$, $g \in G$ and $ a\in A$, is an algebra-isomorphism. 

 Now, let $T \in B(\Sigma)$, so $T=T_{\rho, v, x, y}$ for some equivariant representation $(\rho,v)$ of $\Sigma$ on a Hilbert $A$-module $X$ and  $x, y \in X$. Then set $\widetilde\rho = \rho$ and define  $\widetilde{v}:G \to \mathcal{I}(X)$ by $\widetilde{v}(g) = {\rm ad}_\rho(w(g))  v(g)$, i.e., for each $g\in G$, 
\[\widetilde{v}(g)x = \big(\rho(w(g))v(g)x\big)\cdot w(g)^*\]
for all $x \in X$. We claim that $(\widetilde\rho,\widetilde{v})$ is an equivariant representation of $\Sigma^w$ on $X$. 

Indeed, let $ g, h\in G,  a \in A$ and $x, y \in X$. Then, using the properties of $(\rho, v)$ repeatedly, we get:
\begin{itemize}
\item[(i)] 
\begin{align*} \widetilde\rho\big(\alpha^w_g(a)\big) \widetilde{v}(g) x 
&=\rho\big(w(g)\alpha_g(a)w(g)^*\big) \Big( \big(\rho(w(g)) v(g)x\big)\cdot w(g)^*\Big) \\
&=\big(\rho(w(g)) \rho(\alpha_g(a)) v(g) x \big)\cdot w(g)^*   \\
&=\big(\rho(w(g)) v(g) \rho(a)  x \big)\cdot w(g)^* \\
&=  \widetilde{v}(g)  \widetilde\rho(a) x,
\end{align*}
\item[(ii)] 
\begin{align*}
& \widetilde{v}(g) \widetilde{v}(h)x =  \big(\rho(w(g)) v(g) \widetilde{v}(h) x \big) \cdot w(g)^*  \\
& =  \Big(\rho(w(g)) v(g) \big( (\rho(w(h)) v(h) x) \cdot w(h)^*  \big) \Big) \cdot w(g)^* \\
& = \Big(\rho(w(g)) \Big( \big( (v(g) \rho(w(h)) v(h) x) \cdot \alpha_g(w(h))^* \big) \Big) \Big) \cdot w(g)^* \\
& = \Big(\rho(w(g)) \big( (\rho(\alpha_g((w(h))) v(g) v(h) x) \cdot \alpha_g(w(h))^* \big) \Big) \cdot w(g)^*  \\
& = \Big( \big( \rho(w(g)) \rho(\alpha_g(w(h))) v(g) v(h) x \big) \cdot \alpha_g(w(h))^* \Big) \cdot w(g)^*  \\
& = \Big( \rho(w(g)) \rho(\alpha_g(w(h))) \big( (\rho(\sigma(g,h)) v(gh) x) \cdot \sigma(g,h)^* ) \big) \Big) \cdot \alpha_g(w(h))^*  w(g)^*  \\
& = \Big( \rho(\sigma^w(g,h)) \rho(w(gh)) v(gh) x \Big) \cdot \sigma(g,h)^*  \alpha_g(w(h))^*  w(g)^*  \\
& = \Big( \rho(\sigma^w(g,h)) \rho(w(gh)) v(gh) x \Big) \cdot w(gh)^* w(gh) \sigma(g,h)^*  \alpha_g(w(h))^*  w(g)^*  \\
& = \Big( \rho(\sigma^w(g,h)) \big( ( \rho(w(gh)) v(gh) x ) \cdot w(gh)^* \big) \Big) \cdot \sigma^w(g,h)^*  \\
& = \Big( \rho(\sigma^w(g,h)) \big( \widetilde{v}(gh)x \big) \Big) \cdot \sigma^w(g,h)^*  \\
& = {\rm ad}_{\widetilde\rho}(\sigma^w(g,h)) \widetilde{v}(gh)x, \\
\end{align*}
\item[(iii)] \begin{align*} \alpha^w_g\big(\langle x  , y \rangle\big)   &=  w(g)  \alpha_g\big(\langle x  , y \rangle\big)  w(g)^*\\
&= w(g) \big\langle v(g) x , v(g) y \big\rangle w(g)^* \\
&= w(g) \big\langle \rho(w(g)) v(g) x, \rho(w(g)) v(g) y \big\rangle w(g)^* \\
&=  \big\langle \big(\rho(w(g)) v(g) x\big)\cdot w(g)^*, \big(\rho(w(g)) v(g) y\big) \cdot w(g)^* \big\rangle \\
&= \big\langle \widetilde{v}(g) x, \widetilde{v}(g) y \big\rangle,
\end{align*}
\item[(iv)]  
\begin{align*}
\widetilde{v}(g)(x \cdot a) & = \big( \rho(w(g)) v(g) (x \cdot a) \big) \cdot w(g)^* \\
& = \rho(w(g)) \big( (v(g) x) \cdot (\alpha_g(a) w(g)^*) \big) \\
& = \rho(w(g)) \big( (v(g)x) \cdot (w(g)^* \alpha_g^w(a)) \big) \\
& = \Big( \rho(w(g))\big((v(g)x) \cdot w(g)^*\big)\Big) \cdot \alpha_g^w(a) \\
& = (\widetilde{v}(g)x) \cdot \alpha_g^w(a),
\end{align*}
\end{itemize} 
as claimed. Now for all $g \in G$ and $a\in A$  we have
\begin{align*}
[\Pi(T)](g,a) &= T_{\rho, v, x, y}\big(g, aw(g)\big)w(g)^* = \big\langle x, \rho(aw(g)) v(g) y\big\rangle  w(g)^*\\
&=  \big\langle x, \big(\rho(a)\rho(w(g)) v(g) y\big)\cdot w(g)^*\big\rangle = \big\langle x, \widetilde\rho(a)\widetilde{v}(g) y\big\rangle\\
&=  T_{\widetilde\rho, \widetilde{v}, x, y} (g,a),
\end{align*}
so we get that $\Pi$ maps $B(\Sigma)$ into $B(\Sigma^w)$ and that $\|\Pi(T)\| \leq \|x\|  \|y\|$. Since this inequality holds for any $\rho,v, x,y$ such that $T= T_{\rho, v, x, y}$, it follows that  $\|\Pi(T)\| \leq \|T\|$. By symmetry, we then see that $\Pi$ restricts to an isometric algebra-isomorphism between $B(\Sigma)$ and $B(\Sigma^w)$. 
It follows that $\Psi := \Pi^{-1}$ is an algebra-isomorphism from $B(\Theta) = B(\Sigma^w)$ onto $B(\Sigma)$ such that 1) holds.  
In passing, we note that one can also easily deduce that $\Pi(T)$ is $\Sigma^w$-positive definite whenever  $T$ is $\Sigma$-positive definite, either by a direct computation, or using what we just have done in combination with the Gelfand-Raikov characterization of positive definiteness (cf.~\cite[Corollary 4.4]{BeCo6}). 

Let now $f\in B(G)$ and consider $T^f \in B(\Sigma)$. Then  we have that 
\[\Pi(T^f)(g, a) = T^f\!\big(g, aw(g)\big)w(g)^* = \big(f(g) aw(g)\big)w(g)^* = f(g)  a = T^f\!(g,a) \]
for all $g\in G$ and $a\in A$, which shows that $\Pi(T^f) = T^f \in B(\Sigma^w)$. Thus it is clear that $\Pi$ restricts to the identity map from $B(G)$ (inside $B(\Sigma)$) into $B(G)$  (inside $B(\Sigma^w)$), hence that $\Psi = \Pi^{-1}$ satisfies 2).

Finally, 
 let $b\in A$ and consider $T^b \in B(\Sigma)$. Then  we have that 
\[\Pi(T^b)(g, a) = T^b\big(g, aw(g)\big)w(g)^* = baw(g)w(g)^* = b a = T^b(g,a) \]
for all $g\in G$ and $a\in A$. Thus it is clear that $\Pi$ restricts to the identity map from $A$ (inside $B(\Sigma)$) into $A$ (inside $B(\Sigma^w)$), hence that $\Psi = \Pi^{-1}$ satisfies 3).
\end{proof}

\begin{remark}
The converse of Theorem \ref{coc-conj-isom} is not  true in general. Indeed, set
 \[Z^2(G, \Toro) = \{ \omega:G\times G\to \Toro \mid  \omega \text{ is a normalized $2$-cocycle on $G$}\}.\] 
Then let $\omega \in Z^2(G, \Toro)$ and consider the systems $\Sigma =(\Complessi, G, {\rm triv}, 1)$ and $\Theta=(\Complessi, G, {\rm triv}, \omega)$, where ${\rm triv}$ denotes the obvious action of $G$ on $\Complessi$. Then we have that $B(\Sigma) = B(G) = B(\Theta)$, but $\Sigma$ is not cocycle  group conjugate to $\Theta$ if $\omega$ is not a coboundary. 
\end{remark}

\begin{remark} \label{coc-pert2}
In order to look for a converse of Theorem \ref{coc-conj-isom} one option is to weaken cocycle group conjugacy as follows. If $\Sigma = (A, G, \alpha, \sigma)$ is a system and  $\omega \in Z^2(G,\Toro)$, then we may regard $\omega$ as a normalized $2$-cocycle for the restriction of $\alpha$ to $Z(A)$ and perturb $\Sigma$ by $\omega$ (cf.~Remark \ref{coc-pert}). Obviously, $\Sigma$ and $\Sigma(\omega)= (A, G, \alpha, \sigma_\omega)$ have then the same equivariant representations, so we have that $B(\Sigma)=B(\Sigma(\omega))$.

If  $\Theta =(B, H, \beta, \theta )$ is another system, let us  say that $\Sigma$ and $\Theta$ are \emph{weakly cocycle group conjugate} if  $\Sigma(\omega)$ is cocycle group conjugate to $\Theta$ for some $\omega \in Z^2(G, \Toro)$. Using Theorem \ref{coc-conj-isom} we get that $B(\Theta)$ is then isomorphic to $B(\Sigma(\omega))= B(\Sigma)$ via an algebra-isomorphism satisfying 1), 2) and 3).
 \end{remark}
 
Let us now assume that  the conclusion of Theorem \ref{coc-conj-isom}  holds.
One may then wonder under which additional requirements it would be possible to conclude that  $\Sigma$ and $\Theta$ are weakly cocycle group conjugate. A result in this direction goes as follows.

By  invoking Walter's theorem recalled in the introduction we get from 2) that $\Psi$ determines an isomorphism $\varphi:G\to H$, while 3) gives that there is a $*$-isomorphism $\phi:A \to B$.
 For each $g \in G$, set $\gamma_g := \phi^{-1}\beta_{\varphi(g)} \phi \in {\rm Aut}(A)$. Then one may check whether  $\gamma_g$ and $\alpha_g$ agree up to inner automorphisms for every $g \in G$. Assume that this happens to be the case, i.e., there exists some normalized map $w:G\to\U(A)$ such that 
 $\gamma_g = {\rm Ad}(w(g)) \alpha_g$ for all $g\in G$.   
 Then, letting $u:G\times G \to \U(A)$ be defined by 
 \[u(g, g')= \phi^{-1}\big(\theta(\varphi(g), \varphi(g'))\big) \quad \text{ for all }  g, g' \in G,\]
 we get a twisted action $(\gamma, u)$ of $G$ on $A$. Define then a map $\omega:G\times G \to \U(A)$ by 
  \[\omega(g, g') := u(g, g')  \sigma^w(g, g')^* \]
  for all $g, g' \in G$. Then, using the two expressions for $\gamma$ and making use of some cocycle identities, one verifies that $\omega$ takes its values in $Z(A)$, and that it is a $2$-cocycle for $\alpha'$ (the restriction of $\alpha$ to $Z(A)$). Since $u = (\sigma^w)_\omega$, it follows that
 \[\Theta = (B, H, \beta, \theta) \sim_{gc} (A, G, \gamma, u) = (A, G, \alpha^w, 
(\sigma^w)_\omega) = \Sigma^w(\omega)\]
 (using notation as in Remark \ref{coc-pert}). 
 Hence, if $A$ (and therefore $B$) has trivial center, we get that $\omega \in Z^2(G, \Toro)$ and
  $\Theta$ is group conjugate to $\Sigma^w(\omega)$, which is exterior equivalent to $\Sigma(\omega)$. Thus, $\Sigma$ and $\Theta$ are weakly group cocycle conjugate in this case. 
  
As a consequence, we obtain the following.
\begin{theorem}
Consider two systems $\Sigma = (A, G, \alpha, \sigma)$ 
and  $\Theta = (B, H, \beta, \theta )$. Assume that there exists an algebraic isomorphism $\Pi: B(\Sigma) \to B(\Theta)$ satisfying that
\begin{equation}\label{eqPi}
\Pi(T)\big(\varphi(g),\phi(a)\big) = \phi\big(T(g,aw(g))w(g)^*\big) \quad \text{ for all} \ g \in G, a\in A,
\end{equation}
for some isomorphism $\varphi: G \to H$, some $*$-isomorphism $\phi: A \to B$
and some map $w: G \to \U(A)$, which also satisfies
\begin{equation}\label{eqPi2}
\Pi(T_{\ell_A,\alpha,x,y}) = T_{\ell_B,\beta,\phi(x),\phi(y)}
\end{equation}
for all $x,y \in A$.
If the center of $A$ is trivial, then $\Sigma$ and $\Theta$ are weakly cocycle group conjugate.
\end{theorem}

\begin{proof}
Using 
(\ref{eqPi}) and (\ref{eqPi2}), one deduces that  $ \phi^{-1}\beta_{\varphi(g)}\phi={\rm Ad}(w(g)) \alpha_g$ for all $g \in G$. We are then in the position to proceed as we did above, and the desired assertion follows at once.
\end{proof}

\section{On Morita equivalent systems}

Let us consider two twisted unital discrete $C^*$-dynamical systems $\Sigma=(A,G,\alpha,\sigma)$ and $\Theta=(B,G,\beta,\theta)$ over the same group $G$. (We will briefly discuss the more general situation in Remark \ref{weak ME}.)   Our main aim in this section is to show that  if $\Sigma$ and $\Theta$ are \emph{Morita equivalent} in the sense of \cite{Bui, Kal}, then the Fourier-Stieltjes algebras $B(\Sigma)$ and $B(\Theta)$ can be determined from each other.  Morita equivalence for (untwisted) $C^*$-dynamical systems goes at least back to \cite{Com}.
For the ease of the reader, we review the definitions of the concepts that we will use.

Following \cite{EKQR}, we say that a right Hilbert $B$-module $Z$ is  a \emph{right Hilbert $A$-$B$ bimodule} if there is a 
homomorphism $\kappa: A \to \Lc_B(Z)$.\footnote{We recall that by our standing assumptions, $\kappa$ is then unit preserving, hence nondegenerate, as required in \cite{EKQR}.}  We set $a\cdot z = \kappa(a)z$ for $a\in A$ and $z \in Z$, and frequently write $_AZ_B$ for $Z$. A \emph{right Hilbert  $A$-$B$ bimodule isomorphism}  $\Phi : _A\!\!Z_B \to _A\!\!W_B$ between two right $A$-$B$ Hilbert bimodules $Z$ and $W$ (or simply an isomorphism, for short)  is a bimodule isomorphism such that $\langle\Phi(z), \Phi(z')\rangle_B = \langle z, z'\rangle_B$ for $z, z' \in _A\!\!Z_B$. Left  Hilbert $A$-$B$ bimodules and their isomorphisms are defined in a similar way.

Let  $_AZ_B$ be a right Hilbert $A$-$B$-bimodule.  
A map $\delta$ from $G$ into $\I(Z)$
(the group of invertible $\Complessi$-linear bounded maps from $Z$ into itself) is called a $(\Sigma,\Theta)$-\emph{compatible action  of $G$ on $_AZ_B$} when the following conditions are satisfied for $g \in G, a \in A$, $z,\zeta \in Z$ and $b \in B$:
\begin{itemize}
\item $\delta(g)(a\cdot z) = \alpha_g(a) \cdot (\delta(g)z)$,
\item $\delta(g)(z\cdot b) = (\delta(g)z)\cdot  \beta_g(b)$,
\item $\delta(g)  \delta(h) z = \sigma(g,h) \cdot (\delta(gh)z) \cdot \theta(g,h)^*$,
\item $\big\langle \delta(g)z,\delta(g)\zeta \big\rangle_B = \beta_g(\langle z,\zeta\rangle_B)$.
\end{itemize}
We will let $S_{\delta, z, \zeta}: G\times A\to B$ be the map defined by 
\[ S_{\delta, z, \zeta}(g,a) = \big\langle z, a\cdot (\delta(g)\zeta)\big\rangle_B\]
for all $g\in G$ and $a\in A$. Clearly, if $g\in G$ is fixed,  the map $a\to S_{\delta, z, \zeta}(g,a)$ from $A$ into $B$ is linear; moreover, it is bounded, since one easily shows that 
\[ \|S_{\delta, z, \zeta}(g,a)\|  \leq \|z\| \|\zeta\| \|a\| \]
for all $a\in A$.

Two $(\Sigma,\Theta)$-compatible actions $\delta$ and $\delta'$ of $G$, acting respectively on  $_AZ_B$ and $_AZ'_B$, are called \emph{equivariantly isomorphic} if there exists an isomorphism of right Hilbert $A$-$B$-bimodules between $_AZ_B$ and $_AZ'_B$ which intertwines $\delta$ and $\delta'$. 

\begin{remark} \label{equirep-action}
If $(\rho,v)$ is an equivariant representation  of $\Sigma$ on a right Hilbert $A$-module $X$, then $X$ is a right Hilbert  $A$-$A$-bimodule (using $\rho$ as the left action of $A$ on $X$)
and $v$ is a $(\Sigma,\Sigma)$-compatible action of $G$ on $_A X_A$. Conversely, if $v$ is a $(\Sigma,\Sigma)$-compatible action of $G$ on a right Hilbert $A$-$A$-bimodule $X$, where the left action of $A$ on $X$ is given by some  homomorphism $\rho:A\to \Lc_A(X)$, then 
 $(\rho, v)$ is an equivariant representation of $\Sigma$ on $X$.
For example, if we consider $A$ as a right Hilbert $A$-$A$-bimodule in the obvious way, then the map $\alpha: G \to \I(A)$ is a $(\Sigma,\Sigma)$-compatible action of $G$ on $_AA_A$, corresponding to the trivial equivariant representation $(\alpha, \ell)$ of $\Sigma$ on $A$. 
\end{remark}

 We recall that a right Hilbert $B$-module $X$ is called \emph{full} when $\overline{\langle X, X\rangle} = B$. Fullness of a left Hilbert $C^*$-module is defined in a similar way. An \emph{$A$-$B$ imprimitivity bimodule} $Z={} _AZ_B$ (sometimes called an equivalence $A$-$B$-bimodule)  is a full right Hilbert $A$-$B$-bimodule w.r.t.~a $B$-valued inner product $\langle \cdot, \cdot \rangle_B$, which is also a full left Hilbert $A$-$B$-bimodule w.r.t.~to an $A$-valued inner product $_A\langle \cdot, \cdot \rangle$, in such a way that 
\[_A\langle z, z'\rangle\cdot z'' = z\cdot \langle z', z''\rangle_B \]
for  all $z, z', z'' \in Z$. It then follows that $\|_A\langle z, z\rangle\| =\|\langle z, z\rangle_B\|$ for all $z \in Z$, hence that the two norms on $Z$ associated to the left and the right inner products coincide.

Following \cite{Bui, Kal}, we say that the two systems $\Sigma$ and $\Theta$ are  \emph{Morita equivalent} when there exist an $A$-$B$ imprimitivity bimodule $Z$ together with a $(\Sigma,\Theta)$-compatible action $\delta$ of $G$ on $Z$; we then write $\Sigma \sim _{(Z,\delta)} \Theta$.
We note that  $\delta$ automatically satisfies
\begin{itemize}
\item $_A\big\langle \delta(g)z,\delta(g)\zeta \big\rangle = \alpha_g(_A\langle z,\zeta\rangle),$
\end{itemize}
 see e.g.~the argument given in  \cite[Remark 2.6 (2)]{EKQR}.

It is easy to check that $\Sigma$ and $\Theta$ are  Morita equivalent whenever they are  cocycle conjugate (see e.g.~\cite[Section 9]{Com} for the untwisted case). 
Moreover, Morita equivalent twisted $C^*$-dynamical systems have Morita equivalent
$C^*$-crossed products (see \cite[Theorem 2.3]{Bui} for the full case, and \cite[Sections 2.5.4 and 2.8.6]{CELY} for the reduced case).
We also mention the  following result, which is probably a part of the folklore on this topic.
\begin{proposition}\label{MR-comm}
Assume that $\Sigma=(A,G,\alpha,\sigma)$ and $\Theta=(B,G,\beta,\theta)$ are Morita equivalent, and that $A$ and $B$ are commutative. Then the action $\alpha$ of {}$G$ on $A$ is  conjugate to the action $\beta$ of $G$ on $B$, i.e., there exists an isomorphism $\phi$ from $A$ onto $B$ which intertwines these actions. 
Moreover, $\Sigma$ is conjugate to the system $(B, G, \beta, \sigma_\phi)$, while $\Theta $ is conjugate to the system $(A, G, \alpha, \theta_{\phi^{-1}})$, where $\sigma_\phi(g, h):= \phi(\sigma(g, h))$ and $\theta_{\phi^{-1}}(g,h):= \phi^{-1}(\theta(g,h))$ for all $g, h \in G$.
\end{proposition}
\begin{proof}
The assumption says that  $\Sigma \sim _{(X,\delta)} \Theta$ for some $A$-$B$ imprimitivity bimodule $Z$ and some $(\Sigma,\Theta)$-compatible action $\delta$ of $G$ on $Z$. In particular,  $A$ and $B$ are  Morita equivalent. As $A, B$ are both commutative, we can then apply \cite[Theorem 2.24]{BCL} to conclude that there is a unique isomorphism $\phi:A \to B$ satisfying that
\begin{equation} \label{phi-prop}
\phi(_A\langle z, z'\rangle ) = \langle z', z\rangle_B \quad \text{for all }  z, z'\in Z,
\end{equation}
and we also have that $a\cdot z = z\cdot \phi(a) $ for all $a\in A$ and $z\in Z$.
Using properties of $\delta$ in combination with  (\ref{phi-prop}) we get
\begin{align*} 
\phi\big(\alpha_g(_A\langle z, z'\rangle) \big)&= \phi\big(_A\langle\delta(g) z, \delta(g)z'\rangle \big)= \langle\delta(g) z', \delta(g)z\rangle_B \\
& = \beta_g( \langle z', z\rangle_B) = \beta_g\big( \phi(_A\langle z, z'\rangle)\big)
\end{align*}
for all $g \in G$ and $z, z'\in Z$. Since $Z$ is full as a left Hilbert $A$-module, it follows that $\phi \alpha_g = \beta_g \phi$ for every $g\in G$, hence that $\alpha$ and $\beta$ are conjugate. This shows the first part of the proposition. The second part follows immediately. 
\end{proof}
In the setting  of Proposition \ref{MR-comm}, 
it is not clear that 
$\Sigma$ and $\Theta$ are conjugate. 
However, this is certainly the case when $\sigma$ and $\theta$ are both trivial:

\begin{corollary}\label{com-M-eq}
Suppose that $(A, G, \alpha)$ and $(B, G, \beta)$ are $($untwisted discrete unital$)$ $C^*$-dynamical systems with both $A$ and $B$ commutative. Then these systems are Morita equivalent if and only if they are conjugate,
in which case the associated Fourier-Stieltjes algebras are isometrically isomorphic.
\end{corollary}
\begin{proof} This follows from Proposition \ref{MR-comm} and Theorem \ref{coc-conj-isom}. \end{proof}

Assume now that $\Omega=(C, G, \gamma, \omega)$ is another twisted discrete unital $C^*$-dynamical system, $\delta$ is a $(\Sigma, \Theta)$-compatible action of $G$ on $_AX_B$ and $\eta$ is a $(\Theta, \Omega)$-compatible action of $G$ on $_BY_C$. 
If $\pi:B\to \Lc_C(Y)$ denotes the left action of $B$ on $Y$, we can form the internal tensor product  $X\otimes_\pi Y$, which is a right Hilbert $C$-module (cf.~\cite{La1}); we will suppress $\pi$ in our notation and denote $X\otimes_\pi Y$ by $X\otimes_B Y$ in the sequel, as is common in the literature. 
Moreover,  $X\otimes_B Y$ can be turned into a right Hilbert $A$-$C$ bimodule, the left action of $A$ on $X\otimes_B Y$ being given on simple tensors by $a\cdot(x\dot\otimes y) = (a\cdot x) \dot\otimes y$, and we can define a 
$(\Sigma, \Omega)$-compatible \emph{product action} $\delta\otimes_B\eta$ of $G$ on $_A(X\otimes_B Y)_C$, which is given on simple tensors by
$(\delta\otimes_B\eta)(g)(x\dot\otimes y) = \delta(g)x  \dot\otimes  \eta(g)y$.  
Indeed, as a sample, consider $g, h \in G$, $x \in X$ and $y \in Y$. Then we have
\begin{align*}
\big((\delta\otimes_B\eta)&(g) (\delta\otimes_B\eta)(h)\big) (x\dot\otimes y) 
= (\delta\otimes_B\eta)(g)\big( \delta(h)x  \dot\otimes  \eta(h)y\big) \\ 
& =\delta(g)\delta(h)x  \dot\otimes  \eta(g)\eta(h)y\\ 
&=\big(\sigma(g,h) \cdot (\delta(gh)x) \cdot \theta(g,h)^*\big) \dot\otimes  \big(\theta(g,h) \cdot (\eta(gh)x) \cdot \omega(g,h)^*\big)\\
 &= \big(\big(\sigma(g,h) \cdot (\delta(gh)x) \cdot \theta(g,h)^*\big)\cdot \theta(g,h)\big)\dot\otimes \big(\big(\eta(gh)y\big) \cdot \omega(g,h)^*\big) \\
 &= \sigma(g,h) \cdot \big(\delta(gh)x \dot\otimes \eta(gh)y\big) \cdot \omega(g,h)^* \\
 &= \sigma(g,h) \cdot \big((\delta\otimes_B\eta)(gh)(x\dot\otimes y)\big) \cdot \omega(g,h)^* 
\end{align*}
Thus, by continuity, it follows that $\delta\otimes_B\eta$ satisfies the third property required for being a $(\Sigma, \Omega)$-compatible action.
The reader will find more details about  this 
construction and its  properties in \cite{EKQR-0, EKQR}. These articles deal with the untwisted case, but it is easy to adapt the proofs to our setting. 
In particular, arguing as in the proof of \cite[Theorem 2.8 and Remark 2.9]{EKQR}, we obtain that  the  following facts hold:
\begin{itemize}
\item Up to equivariant isomorphism, the product of compatible actions is associative. 
\item Recalling that $\alpha$ is  a  $(\Sigma, \Sigma)$-compatible action of $G$ on $_AA_A$, the $(\Sigma, \Theta)$-compatible product action $\alpha\otimes_A\delta$ of $G$ on $_A(A\otimes_AX)_B$ is equivariantly isomorphic to $\delta$.  In a similar way,  the product action $\delta\otimes_B\beta$ of $G$ on $_A(X\otimes_B B)_B$ is  equivariantly isomorphic to $\delta$.  
\item Assume that  $\Sigma$ and $\Theta$ are  Morita equivalent with $\Sigma \sim _{(Z,\delta)} \Theta$. Then we have: 
 \begin{itemize}
 \item 
 $\Theta \sim _{(\widetilde{Z},\widetilde\delta)} \Sigma$, where $\widetilde{Z}$ is the  right Hilbert $B$-$A$ bimodule conjugate (or reverse) to $Z$ and $\widetilde{\delta}$ is the $(\Theta, \Sigma)$-compatible action of $G$ on $\widetilde{Z}$ given by $\widetilde\delta(g) \widetilde{z} = \widetilde{\delta(g) z}$. 
 
 \item The product action $\delta\otimes_B \widetilde{\delta}$ of $G$ on $_A(Z\otimes_B\widetilde{Z})_A$ is  equivariantly isomorphic, as a  $(\Sigma, \Sigma)$-compatible action, to  $\alpha$.
 \item The product action $\widetilde\delta\otimes_A \delta$ of $G$ on $_B(\widetilde{Z}\otimes_A Z)_B$ is equivariantly isomorphic, as a $(\Theta, \Theta)$-compatible action, to $\beta$.
 \end{itemize}
 \end{itemize}
Next, consider a  $(\Sigma, \Sigma)$-compatible action $v$  of $G$ on a right Hilbert  $A$-$A$ bimodule $X$. We 
 will use the same notation as in \cite{EKQR} and  
 let $[X,v]$ denote the class of all pairs 
  $(X', v')$ where  $v'$ is a $(\Sigma, \Sigma)$-compatible action of $G$ on a right Hilbert  $A$-$A$-module $X'$ such that $v'$ is equivariantly isomorphic to $v$. 
 Further, we will let $\mathcal{A}(\Sigma)$ denote the collection of these equivalence classes. 
 Using the above properties, one sees that $\mathcal{A}(\Sigma)$ can be equipped with an associative product given by 
 \[ [X_1, v_1]  [X_2,v_2] := [ X_1\otimes_A X_2, v_1\otimes_A v_2],\]
and that $[A, \alpha]$ acts as a unit in  $\mathcal{A}(\Sigma)$.
Moreover, one readily gets the following result. 
 
 \begin{proposition} \label{Mor-eq-action} Assume that the systems $\Sigma$ and $\Theta$ are  Morita equivalent with $\Sigma \sim _{(Z,\delta)} \Theta$,
 and let $v$ be a $(\Sigma, \Sigma)$-compatible action  on a right Hilbert $A$-$A$ bimodule $X$. 
 
   Then $w:=(\widetilde{\delta}\otimes_A v)\otimes_A \delta$ is a $(\Theta, \Theta)$-compatible action  on the right Hilbert $B$-$B$-bimodule $Y:= (\widetilde{Z}\otimes_AX)\otimes_A Z$. 
  
      Moreover, the action $\delta\otimes_B (w\otimes_B \widetilde{\delta})$ on  the right Hilbert $A$-$A$-bimodule
$ \widetilde{Z} \otimes_B  (Y\otimes_B Z)$ is equivariantly isomorphic to $v$. 
 
   Hence, the map $[X,v] \mapsto [Y,w]$ gives a one-to-one correspondence between $\mathcal{A}(\Sigma)$ and $ \mathcal{A}(\Theta) $ which preserves products. 
  \end{proposition}
   
 Taking into account Remark \ref{equirep-action} this result says that, up to 
isomorphism,  the equivariant representations of two Morita equivalent systems are in a one-to-one correspondence. As we will soon see, this has some relevance for the associated Fourier-Stieltjes algebras. By isomorphism of equivariant representations of a system, we mean the following.

Let $(\rho,v), (\rho',v')$ be equivariant representations  of $\Sigma$ on right Hilbert $A$-modules $X$ and $X'$, respectively. Then $(\rho,v)$ and $ (\rho',v')$ are said to be
\emph{isomorphic}
if $v$ and $v'$ are equivariantly isomorphic as $(\Sigma,\Sigma)$-compatible actions of $G$, i.e., there exists an isomorphism of right  Hilbert $A$-modules  $\phi: X \to X'$ which intertwines $v$ and $v'$, as well as  $\rho $ and $\rho'$. 
We note that in this case we have 
\begin{equation}  T_{\rho, v, x, y} = T_{\rho', v', \phi(x), \phi(y)} \end{equation}
for all $x, y \in X$. Indeed, for  each $a\in A$ and $g\in G$, we have
\begin{align*}
 T_{\rho, v, x, y} (g, a) &= \big\langle x, \rho(a)v(g) y\big\rangle = \big\langle \phi(x), \phi\big(\rho(a)v(g)y\big) \big\rangle' \\
 &= \big\langle \phi(x), \rho'(a)v'(g)\phi(y) \big\rangle' =  T_{\rho', v', \phi(x), \phi(y)} (g, a).
 \end{align*}

The following notation will be useful. If $S:G\times A \to B, T:G\times A \to A$ and $R: G\times B \to A$ are maps, then we let $S\cdot T: G\times A \to B$ and $T\cdot R: G\times B\to A$ be the maps given by 
\[ (S\cdot T)(g, a) = S(g, T(g,a)), \quad (T\cdot R)(g, b)= T(g, R(g,b))\]
for all $g\in G, a\in A$ and  $b\in B$. Moreover, we let $S\cdot T \cdot R: G\times B\to B$ be given by \[S\cdot T \cdot R:= (S\cdot T) \cdot R= S \cdot (T\cdot R).\]

\begin{proposition} \label{Morita1} Assume that the systems $\Sigma$ and $\Theta$ are  Morita equivalent with $\Sigma \sim _{(Z,\delta)} \Theta$, and  let $(\rho,v)$ be an equivariant representation of $\Sigma$ on a right Hilbert $A$-module $X$. 
Let $x, x' \in X$ and  $z, z', \zeta, \zeta' \in Z$.   
Then the map
\[S_{\delta, z', \zeta'} \cdot T_{\rho, v, x, x'} \cdot S_{\tilde{\delta}, \tilde{z}, \tilde{\zeta}}:G\times B \to B\]
belongs to $B(\Theta)$. 
Thus 
we get a linear map $F_{z, z', \zeta, \zeta'} : B(\Sigma) \to B(\Theta)$ given by 
\[ F_{z, z', \zeta, \zeta'}(T) = S_{\delta, z', \zeta'}\cdot T \cdot S_{\tilde{\delta}, \tilde{z}, \tilde{\zeta}}\]
for every $T\in B(\Sigma)$.
Similarly, the assignment  $T' \mapsto S_{\tilde{\delta}, \tilde{z}, \tilde{\zeta}}\cdot T'\cdot S_{\delta, z', \zeta'}$ gives a linear map from $B(\Theta)$ into $B(\Sigma)$. 
\end{proposition}  
\begin{proof}
 Let $Y= (\widetilde{Z}\otimes_AX)\otimes_A Z$ and  $w=(\widetilde{\delta}\otimes v)\otimes \delta: G \to \mathcal{I}(Y)$ be as in Proposition \ref{Mor-eq-action}, and let  $\tau:B \to \Lc_B(Y)$ denote the homomorphism coming from the left action of $B$ on $Y$, so $(\tau, w)$ is an equivariant representation of $\Theta$ on the  right Hilbert $B$-module $Y$.

Let  $g\in G$ and $b \in B$. Then we have 
\begin{equation*} \label{eq-imp-map} T_{\tau, w, (\tilde z\dot\otimes x) \dot \otimes z'\:, (\tilde\zeta\dot\otimes x') \dot \otimes \zeta'} (g, b) =
\Big\langle z' , T_{\rho, v, x, x'}\big( g, _A\big\langle z\cdot b, \delta(g)\zeta\big\rangle\big) \cdot \delta(g)\zeta'\Big\rangle_B.
\end{equation*}
Indeed, 
\begin{align*}
T_{\tau, w,  (\tilde z\dot\otimes x) \dot \otimes z', (\tilde\zeta \dot\otimes x') \dot \otimes \zeta'} &(g, b) 
=  \Big\langle (\tilde z\dot\otimes x) \dot \otimes z',\tau(b)w(g) (\tilde\zeta\dot\otimes x') \dot \otimes \zeta'\Big\rangle_B\\
&=  \Big\langle (\tilde z\dot\otimes x) \dot \otimes z', \big( ((\delta(g)\zeta)\cdot b^*)^{\widetilde{}} \dot\otimes v(g)x'\big) \dot \otimes \delta(g)\zeta'\Big\rangle_B\\
&=  \Big\langle  z', \big\langle \tilde z\dot\otimes x,((\delta(g)\zeta)\cdot b^*)^{\widetilde{}} \dot\otimes v(g)x'\big\rangle_A \cdot \delta(g)\zeta'\Big\rangle_B\\
&=  \Big\langle  z', \big\langle  x,\langle \tilde z, ((\delta(g)\zeta)\cdot b^*)^{\widetilde{}}\rangle_A \cdot v(g)x'\big\rangle_A \cdot \delta(g)\zeta'\Big\rangle_B\\
&=  \Big\langle  z', \big\langle  x, _A\langle z, (\delta(g)\zeta)\cdot b^*\rangle \cdot v(g)x'\big\rangle_A \cdot \delta(g)\zeta'\Big\rangle_B\\
&=  \Big\langle  z', \big\langle  x, \rho\big(_A\langle z\cdot b, \delta(g)\zeta\rangle\big) v(g)x'\big\rangle_A \cdot \delta(g)\zeta'\Big\rangle_B\\
&=\Big\langle z' , T_{\rho, v, x, x'}\big( g, _A\big\langle z\cdot b, \delta(g)\zeta\big\rangle\big) \cdot \delta(g)\zeta'\Big\rangle_B,
\end{align*}
as asserted.  Since
\begin{align*}
S_{\tilde{\delta}, \tilde{z}, \tilde{\zeta}} (g,b)&= \big\langle \tilde z, b \cdot (\tilde\delta(g) \tilde\zeta)\big\rangle_A = 
\big\langle \tilde z, \big((\delta(g)\zeta)\cdot b^*\big)\widetilde{}\big\rangle_A \\
&= {}_A\big\langle  z, (\delta(g)\zeta)\cdot b^*\big\rangle =  {}_A\big\langle  z\cdot b, \delta(g)\zeta\big\rangle,
\end{align*}
we get that
\begin{align*}
\big(S_{\delta, z', \zeta'} \cdot T_{\rho, v, x, x'} \cdot S_{\tilde{\delta}, \tilde{z}, \tilde{\zeta}}\big)(g,b)&= 
S_{\delta, z', \zeta'}\big(g, T_{\rho, v, x, x'}\big(g, {}_A\big\langle  z\cdot b, \delta(g)\zeta\big\rangle\big)\big)\\
&=\big\langle z' , T_{\rho, v, x, x'}\big( g, _A\big\langle z\cdot b, \delta(g)\zeta\big\rangle\big) \cdot \delta(g)\zeta'\big\rangle_B.
\end{align*}
This shows  that \[S_{\delta, z', \zeta'} \cdot T_{\rho, v, x, x'} \cdot S_{\tilde{\delta}, \tilde{z}, \tilde{\zeta}} =T_{\tau, w,  (\tilde z\dot\otimes x) \dot \otimes z', (\tilde\zeta \dot\otimes x') \dot \otimes \zeta'} \in B(\Theta)\]
and the first claim follows. The remaining claims are then easily obtained.  
\end{proof}

\begin{corollary} \label{M-eq-FS}
Assume $\Sigma$ and $\Theta$ are  Morita equivalent with $\Sigma \sim _{(Z,\delta)} \Theta$. 
Then $B(\Theta)$ can be determined from $B(\Sigma)$ and $Z$ $($and 
similarly for the other way around$)$. Indeed, we have
\begin{equation}\label{BT}
 B(\Theta) = {\rm Span} \Big\{ F_{z, z', \zeta, \zeta'}(T) \mid T\in B(\Sigma), z, z', \zeta, \zeta' \in Z\Big\}. 
\end{equation}
\end{corollary}

\begin{proof} 
Using Proposition \ref{Morita1} we get that  the right-hand side of (\ref{BT}) is contained in $B(\Theta)$.   To show the reverse inclusion, we first observe that for $z, z', \zeta, \zeta' \in Z$, $g \in G$ and $b\in B$ we have
\begin{align*}
\big(S_{\delta, z', \zeta'}  \cdot S_{\tilde{\delta}, \tilde{z}, \tilde{\zeta}}\big)(g,b)&= 
S_{\delta, z', \zeta'} \big(g, S_{\tilde{\delta}, \tilde{z}, \tilde{\zeta}}(g,b)\big) \\
&=S_{\delta, z', \zeta'} \big(g, {}_A\big\langle  z\cdot b, \delta(g)\zeta\big\rangle\big)\\
&= \big\langle  z', {}_A\big\langle  z\cdot b, \delta(g)\zeta\big\rangle\cdot (\delta(g)\zeta') \big\rangle_B\\
&= \big\langle  z',   z\cdot b \cdot \big\langle \delta(g)\zeta, \delta(g)\zeta' \big\rangle_B\big\rangle_B\\
&=  \big\langle  z',   z\big\rangle_B  b   \big\langle \delta(g)\zeta, \delta(g)\zeta' \big\rangle_B\\
&=  \big\langle  z,   z'\big\rangle_B^{ *}   b   \beta_g\big(\langle \zeta, \zeta' \rangle_B\big)
\end{align*}
Now, since $Z$ is full as a right Hilbert $B$-module, we can use  Lemma 2.5 in \cite{BCL} to find $z_1, z'_1, \ldots, z_n, z'_n\in Z$ such that \begin{equation} \label{unit} 
\sum_{i=1}^n \langle z_i, z'_i\rangle_B = 1_B\quad \text{(the unit of $B$)}.
\end{equation}
(In fact, proceeding as in \cite[p. 90]{GBVF}, one may even choose $z'_j= z_j$ for all $j=1, \ldots, n$, but we won't need this).
We note that 
\begin{equation}\label{Identity} \sum_{i, j=1}^n F_{z_i, z_i', z_j, z_j'}(I_\Sigma) =  \sum_{i, j=1}^n S_{\delta, z_i', z_j'}  \cdot S_{\tilde{\delta}, \tilde{z_i}, \tilde{z_j}} = I_\Theta.
\end{equation}
Indeed, for $g\in G$ and $b\in B$, using (\ref{unit}), we get
\begin{align*}
\Big(\sum_{i, j=1}^n S_{\delta, z_i', z_j'}  \cdot S_{\tilde{\delta}, \tilde{z_i}, \tilde{z_j}}\Big)(g,b) &= 
\sum_{i, j=1}^n \big\langle  z_i,   z_i'\big\rangle_B^{ *}   b   \beta_g\big(\langle z_j, z_j' \rangle_B\big)\\
&= \big(\sum_{i=1}^n \big\langle  z_i,   z_i'\big\rangle_B\big)^{*}   b   \beta_g\big(\sum_{j=1}^n \langle z_j, z_j' \rangle_B\big)= b.
\end{align*}
Let $T'\in B(\Theta)$. For each $i,j,k,l \in \{1, \ldots, n\}$, set 
\[T'_{i,j,k, l} := S_{\tilde{\delta}, \tilde{z_i}, \tilde{z_j}}\cdot T'\cdot S_{\delta, z_k', z_l'},\]
which belongs to $B(\Sigma)$ (by Proposition \ref{Morita1}). Then, using (\ref{Identity}), we get that 
\begin{align*} \sum_{i,j,k,l=1}^n F_{z_k, z_i', z_l, z_j'} & \big(T'_{i,j,k,l}\big) =  \sum_{i,j,k,l}^n  S_{\delta, z_i', z_j'}  \cdot S_{\tilde{\delta}, \tilde{z_i}, \tilde{z_j}}\cdot T'\cdot S_{\delta, z_k', z_l'} \cdot S_{\tilde{\delta}, \tilde{z_k}, \tilde{z_l}}\\
&=\Big(\sum_{i,j=1}^n  S_{\delta, z_i', z_j'}  \cdot S_{\tilde{\delta}, \tilde{z_i}, \tilde{z_j}}\Big) \cdot T'\cdot \Big(\sum_{k,l=1}^n S_{\delta, z_k', z_l'} \cdot  S_{\tilde{\delta}, \tilde{z_k}, \tilde{z_l}}\Big)\\
&= I_\Theta \cdot T' \cdot I_\Theta = T',
\end{align*}
which shows that $T' \in {\rm Span} \Big\{ F_{z, z', \zeta, \zeta'}(T) \mid T\in B(\Sigma), z, z', \zeta, \zeta' \in Z\Big\}$, as desired.
\end{proof}

In view of the last statement of Corollary \ref{com-M-eq}, one might wonder under which assumptions the Fourier-Stieltjes algebras associated to Morita equivalent systems are actually (isometrically) isomorphic, cf.~Theorem \ref{coc-conj-isom} (see also Remark \ref{triv}).
 Also, it would be interesting to investigate whether in general those Fourier-Stieltjes algebras could be Morita equivalent as Banach algebras in some suitable sense (see e.g.~\cite{Gro} or \cite{Par}). 
However, elaborating on this topic would require the development of additional machinery, and we won't discuss this here.

\begin{remark}\label{triv} It may be worth to point out that in general Morita equivalence of systems is not sufficient to ensure that the associated Fourier-Stieltjes algebras are isomorphic. Indeed, consider $\Sigma = (\Complessi, G, {\rm triv}, 1)$ and $\Theta = (M_2(\Complessi), G, {\rm triv}, 1)$ for some discrete group $G$ (where ${\rm triv}$ denotes the trivial action in both cases). It is then easy to see that $\Sigma$ and $\Theta$ are Morita equivalent. On the other hand, $B(\Sigma) = B(G)$ is commutative, while $B(\Theta)$ is not as it contains a copy of $M_2(\Complessi)$.

\end{remark} 

\begin{remark} \label{weak ME} Consider two systems $\Sigma=(A,G,\alpha,\sigma)$ and $\Theta=(B,H,\beta,\theta)$ where $H$ might be different from $G$, as in the previous section. If $\varphi: G\to H$ is an isomorphism, we obtain a new system $\Theta^\varphi= (B, G, \beta^\varphi, \theta^\varphi)$ by setting $\beta^\varphi_g = \beta_{\varphi(g)}$ and $\theta^\varphi(g, g') = \theta(\varphi(g), \varphi(g'))$. One easily checks that $B(\Theta)$ is isometrically isomorphic to $B(\Theta^\varphi)$.  Now, let us  say that $\Sigma$ and $\Theta$ are 
\emph{weakly Morita equivalent} if 
there exist some $\omega \in Z^2(G, \Toro)$ and some isomorphism $\varphi: G\to H$
such that 
$\Sigma(\omega)$ is Morita equivalent to $\Theta^\varphi$.  Corollary \ref{M-eq-FS} gives then that $B(\Sigma) = B(\Sigma(\omega))$ can be determined from $B(\Theta^\varphi)$, hence from $B(\Theta)$, and vice-versa. Finally, we mention that  
 $\Sigma$ and $\Theta$ are 
 weakly Morita equivalent 
  whenever they are cocycle group conjugate, as the reader will easily verify.
\end{remark}

\section{An application to amenable systems} 
Amenability is an important topic within operator algebras, and it has received a good deal of attention, also in connection with $C^*$-dynamical systems (see e.g.~\cite{AD1, AD2, Ex, ExNg, BrOz, BeCo3, Ex4, BeCo6, MSTT, BEW, ABF2, BEW2} and references therein).
Using the technique used in the proof of  Corollary \ref{M-eq-FS}, we will show that amenability of a system, as defined in \cite{BeCo6}, is preserved under Morita equivalence.  
As before, we let $\Sigma=(A,G,\alpha,\sigma)$ and $\Theta=(B,G,\beta,\theta)$ denote two twisted unital discrete $C^*$-dynamical systems. 
We recall that $\Sigma$ is said to be 
\emph{amenable} whenever there exists a net $\{T^{\nu}\} $ in $P(\Sigma)$ such that 
\begin{itemize}
\item each $T^{\nu}$ is finitely supported, i.e., the set $\{ g\in G \mid T^{\nu}_g \neq 0\}$ is finite for each $\nu$,
\item $\{T^{\nu}\} $ is uniformly bounded, i.e., $\sup_{\nu} \|T^{{\nu}}\|_\infty 
< \infty,$
\item $\lim_{\nu} \|T^{\nu}_g(a) - a\| = 0$ for every $g \in G$ and $a\in A$. 
\end{itemize}

Assume for example that  $\Sigma$ has  \emph{Exel's $($positive$)$ approximation property} \cite{Ex, Ex4, ExNg}, 
that is,  there exists a net  $\{\xi_\nu\}$ of finitely supported functions from $G$ into $A$ such that  
\begin{itemize}
\item[(a)] $\sup_\nu \big\| \sum_{g\in G}  \xi_\nu(g)^*\xi_\nu(g)\big\| < \infty$;

\item[(b)] 
$\lim_\nu \big\|  \sum_{h\in G}  \xi_\nu(h)^*a\alpha_g\big(\xi_\nu(g^{-1}h)\big)  - a\big\| = 0
$ for all $g \in G$ and $a \in A$.
\end{itemize}
Then $\Sigma$ is amenable because  setting $T^\nu_g(a) = \sum_{h\in G}  \xi_\nu(h)^*a\alpha_g\big(\xi_\nu(g^{-1}h)\big)$ for all $g\in G$ and $a\in A$
gives a net $\{T^\nu\}$ satisfying the required properties.  
Note that if all $\xi_\nu$'s take their values in  $Z(A)$, then (b) is equivalent to  
\[\lim_\nu \big\|  \sum_{h\in G}  \xi_\nu(h)^*\alpha_g\big(\xi_\nu(g^{-1}h)\big)  - 1_A\big\| = 0\]
 for all $g \in G$. Thus it readily follows that if $\sigma =1$, then $\Sigma$ is amenable whenever the action $\alpha$ is amenable in the sense of \cite{BrOz}, a notion that is stronger than Anantharaman-Delaroche's original definition of amenability of $\alpha$ in \cite{AD1}. 
Notice also that as long as $\sigma$ is scalar-valued then the amenability of $\Sigma$ does not depend on $\sigma$.  
As shown in \cite[Theorem 4.6]{BeCo6}, amenability of $\Sigma$ implies that $\Sigma$ is \emph{regular}, i.e., the full and the reduced  $C^*$-crossed products associated to $\Sigma$ are canonically isomorphic. 
Several other notions of amenability (for untwisted systems) are discussed in \cite{BEW, BEW2}. We note that if $A$ is commutative, $G$ is exact and $\sigma =1$, then it follows readily from \cite[Theorem 5.2]{BEW} that all existing notions of amenability for $\Sigma$ (including ours, and regularity) are equivalent.
  
 Strong and weak equivalence of Fell bundles over groups are studied in \cite{AF, ABF1, ABF2}. Having in mind that $\Sigma$ gives rise to a Fell bundle over $G$ in a canonical way (cf.~\cite{Ex3}), one may for instance deduce from \cite[Corollary 4.5]{ABF1} and  \cite[Theorem 6.23]{ABF2} that regularity and Exel's approximation property are preserved under Morita equivalence of systems. We prove below that this is also true for amenability in our sense. 

\begin{theorem}\label{Amen}  Assume that the systems $\Sigma$ and $\Theta$ 
are  Morita equivalent, with $\Sigma \sim _{(Z,\delta)} \Theta$.  Then  $\Theta$ is amenable whenever  $\Sigma$ is amenable.
 \end{theorem}
 \begin{proof}  Assume that 
 $\Sigma$ is amenable.
 As in the proof of Proposition \ref{Morita1}, we can
 find $z_1, z'_1, \ldots, z_n, z'_n\in Z$ such that \[\sum_{i=1}^n \langle z_i, z'_i\rangle_B = 1_B.\] For later use, we  set $K= \big(\sum_{i=1}^n \|z_i\|  \|z_i'\|)^2$. Let then $F:B(\Sigma) \to B(\Theta)$ be the linear map given  by 
 \[ F=\sum_{i,j=1}^n F_{z_i, z'_i, z_j, z'_j}.\]
 We first note that $F$ maps $P(\Sigma) $ into $P(\Theta)$. To show this, we use the notation introduced in the proof of Proposition \ref{Morita1}. 
Let $T= T_{\rho, v, x, x} \in P(\Sigma)$ and set $y:= \sum_{i=1}^n(\tilde z_i\dot\otimes x) \dot \otimes z'_i\in Y:= (\tilde Z \otimes_A X) \otimes_A Z$. Then we have 
 \begin{align*} F(T) &= \sum_{i,j=1}^n F_{z_i, z'_i, z_j, z'_j}(T_{\rho, v, x, x}) = 
 \sum_{i,j=1}^n T_{\tau, w, (\tilde z_i\dot\otimes x) \dot \otimes z'_i, (\tilde z_j\dot\otimes x) \dot \otimes z'_j}\\& = \sum_{j=1}^n T_{\tau, w, y, (\tilde z_j\dot\otimes x) \dot \otimes z'_j} =  T_{\tau, w, y, y} \in P(\Theta).
\end{align*}
This computation also gives that 
\begin{align*} 
\|F(T)\|_\infty &= \| \langle y, y\rangle_B\| = \|y\|^2 \leq \Big( \sum_{i=1}^n \| (\tilde z_i\dot\otimes x) \dot \otimes z'_i\|\Big)^2 \\
&\leq \Big( \sum_{i=1}^n \|\tilde z_i\|\| x\|\| z'_i\|\Big)^2 = K  \|x\|^2 = K  \|T\|_\infty
\end{align*}
(since $\|\tilde{z}\| = \|z\|$ for all $z\in Z$).
Further,  we note  that $F(T)$ is easily seen to be finitely supported whenever $T\in B(\Sigma)$ is finitely supported. 

Let now $\{T^{\nu}\}$ be a net in $P(\Sigma)$  witnessing the amenability of $\Sigma$. Then $\{F(T^{\nu})\}$ is clearly a net of finitely supported elements in $P(\Theta)$. Moreover, we have that
\[\sup_\nu \|F(T^{\nu})\|_\infty  \leq  K \sup_\nu \|T^{\nu}\|_\infty <  \infty, \]
so  $\{F(T^{\nu})\}$ is uniformly bounded. Finally, let $g \in G$ and $b\in B$. Then 
\begin{align*}
F(T^{\nu})(g, b) &= \sum_{i,j=1}^n \big(F_{z_i, z'_i, z_j, z'_j}(T^{\nu})\big)(g,b) 
=  \sum_{i,j=1}^n \big(S_{\delta, z_i', z_j'}  \cdot T^{\nu} \cdot S_{\tilde{\delta}, \tilde{z_i}, \tilde{z_j}}\big)(g,b)\\
&=  \sum_{i,j=1}^n S_{\delta, z_i', z_j'}\big(g, T^{\nu}_g\big(S_{\tilde{\delta}, \tilde{z_i}, \tilde{z_j}}(g,b)\big)\big).
\end{align*} 
Consider now $i,j\in \{1, \ldots, n\}$. Using that the map $a \to  S_{\delta, z_i', z_j'}\big(g, a\big)$ from $A$ into $B$ is continuous, we get that
\begin{align*} \lim_\nu  S_{\delta, z_i', z_j'}\big(g, T^{\nu}_g\big(S_{\tilde{\delta}, \tilde{z_i}, \tilde{z_j}}(g,b)\big)\big) 
&=  S_{\delta, z_i', z_j'}\big(g, \lim_\nu T^{\nu}_g\big(S_{\tilde{\delta}, \tilde{z_i}, \tilde{z_j}}(g,b)\big)\big)\\
&=S_{\delta, z_i', z_j'}\big(g, \big(S_{\tilde{\delta}, \tilde{z_i}, \tilde{z_j}}(g,b)\big)\big)\\
&= \big(S_{\delta, z_i', z_j'} \cdot  S_{\tilde{\delta}, \tilde{z_i}, \tilde{z_j}}\big)(g,b).
\end{align*}
Hence, using Equation \eqref{Identity}, we get that 
\begin{align*} \lim_\nu F(T^{\nu})(g, b) &= \sum_{i,j=1}^n \lim_\nu S_{\delta, z_i', z_j'}\big(g, T^{\nu}_g\big(S_{\tilde{\delta}, \tilde{z_i}, \tilde{z_j}}(g,b)\big)\big)\\
&= \sum_{i,j=1}^n \big(S_{\delta, z_i', z_j'} \cdot  S_{\tilde{\delta}, \tilde{z_i}, \tilde{z_j}}\big)(g,b) = I_{\Theta}(b) = b.
\end{align*}
This shows that $\Theta$ is amenable, as desired. 
\end{proof} 

An immediate consequence of this result is that amenability of a system is also preserved under weak Morita equivalence (as defined in Remark \ref{weak ME}).  
\begin{remark} 
A result of a nature similar to Theorem \ref{Amen} is Theorem 2.2.17 in \cite{ADR}, which says that topological amenability of locally compact groupoids is invariant under topological equivalence (whose definition is hinted by Morita equivalence of $C^*$-algebras). 
\end{remark}

\subsection*{Acknowledgements}
Most of the present work has been done during the visits
made by E.B. at the Sapienza University of Rome 
and by R.C. at the University of Oslo in the period 2018-2019. 
Both the authors are grateful to the hosting institutions for the kind hospitality,
and to the Trond Mohn Foundation (TMS) for financial support. 
We are also indebted to the referees for their many helpful comments and suggestions, 
leading to an improved version of this article.

\end{document}